\documentclass[a4paper,10pt]{amsart}
\usepackage[utf8]{inputenc}
\usepackage{amsthm}
\usepackage{amsmath}
\usepackage{bm}
\usepackage{amsfonts}
\usepackage{amssymb}
\usepackage{mathtools}
\usepackage{appendix}
\usepackage{mathrsfs}
\usepackage{setspace}
\usepackage{thmtools} % solves problem of shared counters in statements and autoref
\usepackage[foot]{amsaddr}
\usepackage[pdfdisplaydoctitle,colorlinks,breaklinks,urlcolor=blue,linkcolor=blue,citecolor=blue]{hyperref} 
\usepackage[nameinlink,capitalise]{cleveref}%must always be the last
\newcommand{\F}{\mathcal{F}}
\newcommand{\G}{\mathcal{G}}
\renewcommand{\H}{\mathbb{H}}

\newcommand{\N}{\mathbb{N}}

\newcommand{\R}{\mathbb{R}}

\DeclareMathOperator{\var}{Var}

\renewcommand{\epsilon}{\varepsilon}

\newcommand{\set}[1]{\left\{#1\right\}}
\newcommand{\pa}[1]{\left(#1\right)}

\newcommand{\abs}[1]{\left|#1\right|}
\newcommand{\norm}[1]{\left\|#1\right\|}

\newcommand{\expt}[2][]{\mathbb{E}_{#1}\left[#2\right]}

\newtheorem{theorem}{Theorem}[section]

\newtheorem{corollary}[theorem]{Corollary}
\newtheorem{lemma}[theorem]{Lemma}

\theoremstyle{remark}
\newtheorem{remark}[theorem]{Remark}

\numberwithin{equation}{section}

\newenvironment{acknowledgements}{%
  \begin{abstract}
}{%
  \end{abstract}
}

\title[Polyspectra of Euclidean and Spherical RWM]{Fluctuations of Polyspectra in Spherical and Euclidean Random Wave Models}
\author[F. Grotto]{Francesco Grotto}
\address{Università di Pisa, Dipartimento di Matematica, 5 Largo Bruno Pontecorvo, 56127 Pisa, Italia.}
\email{\href{mailto:francesco.grotto at unipi.it}{francesco.grotto at unipi.it}}
\author[L. Maini]{Leonardo Maini}
\address{Université du Luxembourg, Maison du Nombre, 6 Avenue de la Fonte, 4364 Esch-sur-Alzette, Luxembourg.}
\email{\href{mailto:leonardo.maini at uni.lu}{leonardo.maini at uni.lu}}
\author[A. P. Todino]{Anna Paola Todino}
\address{Sapienza Università di Roma, Dipartimento di Scienze Statistiche, Piazzale Aldo Moro 5, 00185 Roma, Italia.}
\email{\href{mailto:annapaola.todino@uniroma1.it}{annapaola.todino at uniroma1.it}}
\date\today

\begin{document}

\begin{abstract}
    We consider polynomial transforms (polyspectra) of Berry's model --the Euclidean Random Wave model-- and of Random Hyperspherical Harmonics.
    We determine the asymptotic behavior of variance for polyspectra of any order in the high-frequency limit. In particular, we are able to treat polyspectra of any odd order $q\geq 5$, whose asymptotic behavior was left as a conjecture in the case of Random Hyperspherical Harmonics by Marinucci and Wigman (\emph{Comm. Math. Phys.} 2014). 
    To this end, we exploit a relation between the variance of polyspectra and the distribution of uniform random walks on Euclidean space with finitely many steps, which allows us to rely on technical results in the latter context.
\end{abstract}

\maketitle

%%%%%%%%%%%%%%%%%%%%%%%%%%%%%%%%%%%%%%%%%%%%%
\section{Introduction and Main Results}\label{sec:introduction}

A \emph{Random Wave} model (RW) can be defined as a random field on a Riemannian manifold given by a random combination of eigenfunctions of the Laplace-Beltrami operator at a fixed frequency, or within a small frequency bandwidth. The study of such stochastic objects dates back to the works of Berry and Zelditch \cite{Berry1977,Zelditch2009} and it
is now a well-established area of research, in which properties of nodal sets of RWs and the comparison with their deterministic counterpart have garnered significant focus (\emph{cf.} \cite{Wigman2022} for a recent survey).
In homogeneous spaces such as $\R^d$, the hypersphere $S^d$ or the hyperbolic space $\H^d$ (\emph{cf.} respectively \cite{Maini2022,Marinucci2011book,Grotto2023} for specific discussions), RWs also appear naturally in the spectral decomposition of isotropic Gaussian random fields, that is Gaussian fields whose law (thus, its covariance function) is invariant under isometries of the underlying manifold, as detailed in \cite{Cohen2012}.

In this note, we focus on Gaussian RWs on Euclidean space $\R^d$ and on the (hyper)sphere $S^d\subset \R^{d+1}$. Our main result, Theorem \ref{thm:main}, establishes asymptotics for (variances of) polynomial transforms of those RWs. In particular, we complete the high-frequency asymptotic description of the so-called \emph{polyspectra}, that is integrals of Hermite polynomials of the RWs on fixed space domains:
the case of Hermite polynomials of large odd order was left as a conjecture on hyperspheres in \cite{Marinucci2014,Marinucci2015} and was not discussed in works (such as \cite{Nourdin2019,Maini2022}) in the Euclidean case (see Remark \ref{rmk:biblio} for a precise comparison with previous results). Polyspectra constitute the elementary objects in the Wiener chaos decomposition of functionals of Gaussian RWs, therefore their asymptotics play a crucial role in the Wiener chaos approach to limit theorems.

In order to prove Theorem \ref{thm:main} we will exploit a peculiar relation with uniform random walks on Euclidean space, and the well-established theory of those random processes. To some extent, this resembles a connection with short random walks considered in the case of random spherical harmonics in \cite{Marinucci2008,Marinucci2010}.

\subsection*{Notation}
    We regard the $d$-dimensional sphere $S^d=\{x\in \R^{d+1}:|x|=1\}$ as an embedded manifold of codimension 1; we denote by $\sigma_d$ the induced volume measure and by
    $\omega_d=\sigma_d(S^d)=2 \pi^{(d+1)/2}/\Gamma\pa{\frac{d+1}{2}}$ the total volume.
    The geodesic distance on $S^d$ is given by $d(x,y)=\cos^{-1}(x\cdot y)$, $x\cdot y$ denoting the Euclidean scalar product in $\R^{d+1}$.

    We denote by $B^E_d(x,R)\subset \R^d$ the ball of radius $R$ in $\R^d$ centered at the point $x$. When the context allows it, we lighten the notation omitting dependencies, writing for instance $B_R$ to denote a ball of radius $R$ whose center can be chosen arbitrarily. The same applies to geodesic balls $B^S_d(x,R)\subset S^d$. The symbol $\chi_A$ denotes the indicator function of a set $A$.

    The symbol $C$ denotes a \emph{positive} constant, possibly differing in any of its occurrences, and depending only on eventual subscripts, as in $C_{a,b}$. Landau's $O$ and $o$ symbols have their usual meaning, with constants involved in upper and lower bounds depending again only on eventual subscripts, as in $O_{a,b}(1)$.

\subsection{Random Wave Models}
\emph{Berry's model}, that is the RW model on Euclidean space $\R^d$,
is the centered Gaussian random field $U_\lambda(x)$, $\lambda>0$, $x\in \R^d$, with covariance function
\begin{equation}\label{eq:covberry}
    \expt{U_\lambda(x)U_\lambda(y)}=
    j_d(\lambda|x-y|),\quad
    x,y\in \R^d,
\end{equation}
where $j_d$ is the Fourier transform of the volume measure $\sigma_{d-1}$ of $S^{d-1}$, regarded as a generalized function on the ambient space $\R^d$,
\begin{equation}\label{eq:jd}
     j_d(|x|) =  \frac1{\omega_{d-1}}\int_{S^{d-1}} e^{i x\cdot \theta} d\sigma_{d-1}(\theta)=
     \frac{\nu!2^\nu}{(|x-y|)^\nu}J_\nu(|x-y|), \quad \nu=\frac{d}{2}-1,
\end{equation}
$J_\nu$ being the Bessel function of first kind and order $\nu$. 
In particular, $j_2(r)=J_0(r)$ and $j_3(r)=\text{sinc}(r)=\frac{\sin(r)}{r}$.

The Gaussian field $U_\lambda$ is named after M. Berry, who introduced it in order to describe local behavior of high frequency eigenstates in quantum billiards \cite{Berry1977,Berry2002a}.
It arises as a central limit of random linear combinations of planar waves on $\R^d$ of the form
\begin{equation*}
    \frac{1}{\sqrt N}\sum_{i=1}^N \cos\pa{\lambda x\cdot y_i+\phi_i}, \quad x\in \R^d,\quad N\to\infty,
\end{equation*}
with the $y_i$'s and $\phi_i$'s being i.i.d. uniform random variables respectively in $S^{d-1}$ and $[0,\pi]$ (\emph{cf.} \cite{Berry2002a,Nourdin2019}). Samples of $U_\lambda$ are themselves smooth $\lambda^2$-eigenfunctions of the Laplace operator.

\emph{Random (Hyper)Spherical Harmonics}, that is the RW model on the sphere $S^d$, consist in the centered Gaussian random field $T_\ell(x)$, $\ell\in \N$, $x\in S^d$, with covariance 
\begin{equation}\label{eq:covrsh}
    \expt{T_\ell(x)T_\ell(y)}=\G_{d,\ell}(\cos d(x,y))
    =\binom{\ell+\nu}{\ell}^{-1}P^{\nu,\nu}_\ell(\cos d(x,y))
    \quad x,y \in S^d,
\end{equation}
where $\G_{d,\ell}$ is
the normalized Gegenbauer polynomial of degree $\ell $ (see \cite[4.7]{Szego1939}), the right-hand side providing an alternative representation in terms of Jacobi polynomials.
This random field can be regarded as the analog of Berry's model on $\R^d$. Indeed, the role of Euclidean planar waves is played on the sphere $S^d$ by hyperspherical harmonics, that is Laplace-Beltrami eigenfunctions
\begin{equation*}
    \set{Y_{\ell,m,d}}_{m=1,\dots,\eta_{\ell,d}}\subset C^\infty(S^d),\, \ell\in\N,\quad 
    \Delta_{S^d} Y_{\ell,m,d}+\ell(\ell+d-1) Y_{\ell,m,d}=0,
\end{equation*}
forming a complete orthonormal basis of $L^2(S^d,\sigma_d)$, spanning at each fixed $\ell\in\N$ a distinct eigenspace\footnote{
In fact, the arbitrary choice of an orthogonal basis of the eigenspace relative to $\ell(\ell+d-1)$ is irrelevant in our scope.
}  of dimension $\eta_{\ell,d}=\frac{2\ell+d-1}{\ell} \binom{\ell+d-2}{\ell-1}$.
%\begin{equation*}
%    \eta_{\ell,d}=\frac{2\ell+d-1}{\ell} \binom{\ell+d-2}{\ell-1} \sim \frac{2 \ell^{d-1}}{(d-1)!}, \quad \ell\to \infty.
%\end{equation*}
The Gaussian field $T_\ell$ has the law of a random superposition of waves at a fixed wavenumber:
\begin{equation*}
    T_\ell(x)\sim \sum_{m=1}^{\eta_{\ell,d}} a_{\ell,m,d} Y_{\ell,m,d}(x), \quad x \in S^d,
\end{equation*}
with $\{a_{\ell,m,d}\}_{m=1}^{\eta_{\ell,d}}$ being i.i.d. Gaussian variables $N\pa{0, \omega_d/\eta_{\ell,d}}$.
The index $\ell$ parametrizes the spectrum of Laplace-Beltrami operator on $S^d$ through $\ell(\ell+d-1)$, thus in the limit $\ell\to\infty$ it plays the same asymptotic role of $\lambda\to \infty$ in the Euclidean setting, where $\lambda^2$ is the eigenvalue of the wave $x\mapsto e^{i\lambda x\cdot \theta}$.

\subsection{Polyspectra}
Relevant geometric functionals of $U_\lambda$, such as the $(d-1)$-dimensional volume of the nodal set $\set{U_\lambda=0}$, depend in general also on derivatives (\emph{i.e.} the gradient) of the random field. However, already in the case of integral functionals of the form
\begin{equation}\label{eq:intfunc}
    \F_\lambda=\int_{D} F(U_\lambda(x))dx,\quad F:\R\to\R,\,\lambda>0,\,D\subset\R^d,
\end{equation}
(and the analog for spherical RWs on measurable subsets $D\subseteq S^d$)
the description of the asymptotic behavior as $\lambda\to\infty$ is not trivial, essentially because covariance functions do not satisfy integrability conditions allowing direct applications of basic results (\emph{e.g.} \cite{Breuer1983}), and they oscillate between positive and negative values.

\begin{remark}
    The covariance function \eqref{eq:covberry} depends on $\lambda|x-y|$, therefore the high-frequency limit for functionals of the form \eqref{eq:intfunc} is equivalent to a (fixed-frequency) large domain limit. Indeed, the random fields $B_\lambda(\cdot)\sim B_{1}(\lambda \cdot)$ have the same law, and thus integral functionals
    \begin{equation*}
        \int_{D} F(U_\lambda(x))dx\sim \frac{1}{\lambda^d}\int_{\lambda D} F(U_1(x))dx,
    \end{equation*}
    are also equidistributed.
    Of course, no large-domain limit can be considered on $S^{d}$.
\end{remark}

Since the RWs we consider are Gaussian random fields, functionals of the form \eqref{eq:intfunc} can be studied considering their Wiener chaos expansion.
Recalling that Hermite polynomials, $H_n(t)=(-1)^n \phi^{(n)}(t)/\phi(t)$, $n\in\N$,
%\begin{equation*}
%    H_n(t)=(-1)^n\frac{\phi^{(n)}(t)}{\phi(t)},\,n\in\N,
%    \quad H_0(t)=1,\quad H_1(t)=t,\quad H_2(t)=t^2-1,
%\end{equation*}
form a orthogonal basis of $(\R,\phi(t)dt)$, $\phi$ being the p.d.f. of the standard Gaussian variable, if $F\in L^2(\R; \phi(x)dx)$ the Wiener chaos decomposition of $\F_\lambda$ is given by
\begin{equation*}
    \F_\lambda=\sum_{q=0}^\infty \frac1{q!}\pa{\int_\R F(t)H_q(t)\phi(t)dt} \pa{\int_D H_q(U_\lambda(x))dx}.
\end{equation*}
The stochastic terms $\int_D H_q(U_\lambda(x))dx$ in the decomposition are called \emph{polyspectra}:
Fourth Moment Theorems --by now a standard tool in Gaussian analysis, \emph{cf.} \cite{Nourdin2012}-- allow to deduce Central Limit Theorems for single polyspectra and functionals of the above form as $\lambda\to\infty$,
provided that asymptotics of variance and fourth cumulants are available.

\subsection{Variance Asymptotics}
For $d,q\geq 2$, $R>0$, we will write:
\begin{gather*}
    V^{E}_{d,R}(q,\lambda)=
        \var\pa{\int_{B^E(x_0,R)} H_q(U_{\lambda}(x))dx},
        \quad \lambda>0,\\
    V^{S}_{d,R}(q,\ell)=
        \var\pa{\int_{B^S(x_0,R)} H_q(T_{\ell}(x))d\sigma_d (x)},
        \quad \ell\in\N.
\end{gather*}

\begin{remark}
    The case $q=0$ needs no discussion. As for $q=1$, it turns out that both $V^{E}_{d,R}(1,\lambda)$ and $V^{S}_{d,R}(1,\ell)$ have an oscillatory behavior as $\lambda,\ell\to\infty$ and they can vanish (see respectively \cite{Maini2022,Todino2019} in the two geometrical settings). Anyways, when $q=1$ the polyspectrum is a Gaussian variable, and the study of its variance is simpler, so we will omit that case in our discussion.
\end{remark}

Our main result is the following:

\begin{theorem}\label{thm:main}
    Let $d,q \geq 2$, $R>0$. There exist finite \textbf{positive} constants $c^d_q\in (0,\infty)$ such that:
    \begin{itemize}
        \item (Euclidean) as $\lambda\to\infty$,
        \begin{equation*}
            V^{E}_{d,R}(q,\lambda)=c^d_q q! \omega_d \omega_{d-1} R^d (1+o(1))\cdot
        \begin{cases}
            \lambda^{1-d} &q=2\\
            \lambda^{-2}\log(\lambda) &q=4,d=2\\
            \lambda^{-d} &\text{all other }d\geq2, q\geq 3
        \end{cases},
        \end{equation*}
        \item (Hyperspherical) as $\ell\to\infty$, if $R\in (0,\pi)$,
        \begin{equation*}
            V^{S}_{d,R}(q,\ell)= c^d_q q! \omega_{d-1} \sigma_d(B_R^S)
            (1+o(1))\cdot
        \begin{cases}
            \ell^{1-d} &q=2,\\
            \ell^{-2}\log(\ell) &q=4,d=2,\\
            \ell^{-d} &\text{all other }d\geq2, q\geq 3,
        \end{cases}
        \end{equation*}
        and when $R=\pi$ (that is in the case of polyspectra obtained integrating over the whole $S^d$),
        \begin{equation*}
            V^{S}_{d,\pi}(q,\ell)= 2c^d_q q! \omega_{d-1} \omega_d
            (1+o(1))\cdot
        \begin{cases}
            \ell^{1-d} &q=2,\\
            \ell^{-2}\log(\ell) &q=4,d=2,\\
            0 &q,\ell \text{ both odd,}\\
            \ell^{-d} &\text{all other }d\geq2, q\geq 3.
        \end{cases}
        \end{equation*}
    \end{itemize}
\end{theorem}

The proof of Theorem \ref{thm:main} is the content of the forthcoming Section \ref{sec:polyspectra}.

\begin{remark}\label{rmk:biblio}
    Let us clarify the relation of Theorem \ref{thm:main} with the previous literature.
    \begin{itemize}
        \item (Euclidean case) \cite{Nourdin2019} proved the result in dimension $d=2$ for even $q\geq 2$ and upper bounds for odd $q\geq 3$, for integration domains $D\subset \R^2$ with boundary of class $C^1$; 
        \cite{Notarnicola2021} generalizes the discussion to arbitrary $d\geq 2$ and even $q\geq 2$;
        \item (case of $S^2$) \cite{Marinucci_2011,Marinucci2014} proved the above asymptotics on the whole $S^2$ (case $R=\pi$) leaving as an open problem the positivity of $c^2_q$ for odd $q\geq 5$; \cite{Todino2019} discussed the case $R<\pi$ for $q=2$ and provided upper bounds for larger $q\geq 3$;
        \item (Hyperspherical case) concerning polyspectra over the whole $S^d$ ($R=\pi$), in larger dimension the above asymptotic was established for all even $q\geq 2$ in \cite{Marinucci2015} and for $q=3$ \cite{Rossi2019}, leaving positivity of $c^d_q$ for odd $q\geq 5$ as a conjecture (see \cite[p. 2386]{Marinucci2015}).
    \end{itemize}
    In the Euclidean case, an extension to more general integration domains (convex $D\subset \R^2$ with finite perimeter) for $d\geq 2$ and even $q\geq 2$ will appear in \cite{Maini2023}, treating a large class of Gaussian random fields obtained as mixtures of Berry's model at different wavenumbers $\lambda$.
\end{remark}

The remainder of this note is organized as follows. Section
\ref{sec:numerelli} details the relation between variance of polyspectra and the density of uniform random walk, and introduces results on the latter that we can exploit to control the former. We prove Theorem \ref{thm:main} in Section \ref{sec:polyspectra}.

%%%%%%%%%%%%%%%%%%%%%%%%%%%%%%%%%%%%%%%%%%%%%%%%%%%%%%%%%%%%%
\section{Pearson's Random Walk}\label{sec:numerelli}

Let $X_n^d$, $n=1,2,\dots$ denote the uniform random walk on $\R^d$,
that is 
\begin{equation*}
    X_n^d=U_1^d+\dots+U_n^d,
\end{equation*}
with $(U_k^d)_{k\in \N}$ a sequence of i.i.d. uniform random variables on $S^{d-1}\subset \R^d$. The stochastic process $X^d_n$ is also known as \emph{Pearson's Random Walk}, because of a very early mention in \cite{Pearson1905}.

Much is known on Pearson's random walk: in particular, its distribution and moments for a small number of steps $n$ are of interest in Number Theory because of their relation with Mahler measures. We can thus rely on the series of works \cite{Borwein2011,Borwein2012,Borwein2013,Borwein2016}, to which we refer for a comprehensive background on the matter and the related literature.

The basic result on the distribution of $X_n$ for small $n$ is
a representation formula for the p.d.f. 
dating back to the seminal work \cite{Kluyver1906}.
%(see also \cite{Kiefer1984} and the monograph \cite{Watson1995} for the relation with the theory of Bessel functions):

\begin{lemma}[Kluyver's Formula]\label{lem:kluyver}
    For $n\geq 2$ the law $\mu^d_n$ of $X_n^d$ is absolutely continuous with respect to Lebesgue's measure on $\R^d$, and its density is a radial function. 
    The (Euclidean) distance from the origin $\norm{X_n^d}$ is absolutely continuous with respect to Lebesgue's measure on $(0,\infty)$, and
    its density is given by
    \begin{equation*}
        \rho^d_n(r)=\frac1{(\nu!)^2 4^\nu} \int_0^\infty (t r)^{2\nu+1}j_d(t r) j_d(t)^n dt, \quad r>0.
    \end{equation*}
\end{lemma}

To be precise, the integral in Kluyver's formula has to be regarded as an improper integral that either converges to a non-negative number (for large enough $d,n$ even absolutely) or diverges to $+\infty$ (see Remark \ref{rmk:singolarita} below).
We refer to \cite[Theorem 2.1]{Borwein2016} for a proof of Lemma \ref{lem:kluyver} and a thorough general discussion. 

\subsection{Random Walks and Polyspectra}
Before we outline further properties of Pearson's random walk to be used in the last Section, let us explain heuristically how one is led to consider this apparently unrelated object in studying RW's polyspectra. The following computation concerns the Euclidean case only, and has no direct analog in the case of hyperspherical RWs. Nevertheless, in the high-frequency limiting regime the Euclidean and hyperspherical RW models turn out to exhibit the same behavior.
It holds
\begin{equation*}
    V^{E}_{d,R}(q,\lambda)
        =q!\int_{B_R}\int_{B_R}j_d( \lambda |x-y|)^qdxdy,
\end{equation*}
since $j_d( \lambda |x-y|)$ is the covariance function of $U_\lambda$, and if 
$X,Y$ are two centered jointly Gaussian variables $\expt{H_q(X)H_q(Y)}=q!\expt{XY}^q$.
Recalling that the function $j_d=\F[\sigma_{d-1}]$ is the Fourier transform of the volume measure of $S^{d-1}\subset \R^d$ (\emph{cf.} \eqref{eq:jd}), we can write
\begin{align*}
    V^{E}_{d,R}(q,\lambda)
    &=q!\int_{\R^d}\int_{\R^d} \frac{\chi_{B_{\lambda R}}(x)\chi_{B_{\lambda R}}(y)}{\lambda^{2d}} \F[\sigma_{d-1}](x-y)^q dxdy\\
    &=q!\int_{\R^d}\int_{\R^d} \frac{\chi_{B_{\lambda R}}(x)\chi_{x+B_{\lambda R}}(z)}{\lambda^{2d}} \F[\underbrace{\sigma_{d-1}\ast \cdots \ast \sigma_{d-1}}_{q \text{ times}}](z) dxdz,
\end{align*}
where we recognize the $q$-fold convolution of $\sigma_{d-1}$ with itself as the law $\mu^d_q$ of $X^d_q$. As $\lambda\to\infty$, indicator functions in the last displayed equation converge pointwise to $1$ on $\R^d$,
therefore we expect ---at least intuitively--- that the limiting integral
\begin{equation*}
    \int_{\R^d} \F[\mu^d_q](z) dz=
    \F^{-1}\F[\mu^d_q](0)=\frac{d\mu^d_q}{dx}(0)
    =\lim_{r\to 0} \frac{\rho^d_{q}(r)}{\omega_{d-1}r^{d-1}}
\end{equation*}
plays a relevant role in determining the asymptotic behavior of $V^{E}_{d,R}(q,\lambda)$. By Kluyver's formula we can also write
\begin{equation*}
    \lim_{r\to 0} \frac{\rho^d_q(r)}{r^{d-1}}=
    \frac{1}{(\nu!)^2 4^\nu} \int_0^\infty j_d(t)^q t^{d-1} dt=
    \rho^d_{q-1}(1),
\end{equation*}
therefore the crucial objects appear to be the integrals
\begin{equation}\label{eq:cdq}
    I^d_q\coloneqq \int_0^\infty j_d(t)^q t^{d-1}dt=\frac{\rho^d_{q-1}(1)}{(\nu!)^2 4^\nu },
\end{equation}
and these numbers can \emph{a priori} take any value in $[0,\infty]$.

In fact, $I^d_q$ \emph{coincides with the coefficient} $c^d_q$ \emph{appearing in} Theorem \ref{thm:main} in the cases $d\geq 2,q\geq 3$ except $d=2,q=4$.
Before we prove it in the last Section, we need to discuss finiteness and positivity of the $I^d_q$'s, that is the task for which we need to resort to previous results on Pearson's random walk.

\subsection{Density of short random walks}
We report an explicit expression for the densities $\rho^d_2$
%of $\norm{X^d_2}$, that is the random walk of $n=2$ steps in dimension $d$,
and a recursion formula allowing to treat the density at higher $n$.
Representations for $\rho^d_n$, $n=3,4$, in terms of hypergeometric functions are available in \cite[Equations (74),(79)]{Borwein2016}.

\begin{lemma}\cite[Lemma 4.1]{Borwein2016}\label{lem:rho2}
    For $d\geq 2$ and $0<r<2$ it holds
    \begin{equation*}
        \rho^d_2(r)=\frac{2}{\pi\binom{2\nu}{\nu}}r^{2\nu}(4-r^2)^{\nu-1/2}
        \chi_{(0,2)}(r).
    \end{equation*}
\end{lemma}

\begin{lemma}\cite[Theorem 2.9]{Borwein2016}\label{lem:recursion}
    Let $d,n\geq 2$, $r>0$ and define $\psi^d_n(r)=\rho^d_n(r)/r^{d-1}$.
    It holds, for $n\geq 3$,
    \begin{equation}\label{eq:recursion}
        \psi^d_n(r)=\frac{(\nu!)^2 4^\nu}{\pi (2\nu)!}
        \int_ {-1}^1 \psi_{n-1}^d\pa{\sqrt{1+2s r+r^2}}(1-s^2)^{\nu-1/2}ds,
        \quad r\in (0,n).
    \end{equation}
\end{lemma}

\begin{corollary}\label{cor:positivi}
   For any $d\geq 2$ and $n\geq 3$, $\psi^d_n$ is strictly positive (possibly infinite) on $(0,n)$.
\end{corollary}

\begin{proof}
    By induction on $n$, the case $n=2$ follows by definition of $\psi_2^d$ and Lemma \ref{lem:rho2}. Assume now that $\psi_{n-1}^d(r)>0$ for $r\in(0,n-1)$; by (\ref{eq:recursion}) if the open set
    \begin{equation*}
        \{s\in(-1,1) : \psi_{n-1}^d\pa{\sqrt{1+2s r+r^2}}>0\}=\{s\in(-1,1) : 1+2s r+r^2<(n-1)^2\}
    \end{equation*}
    is not empty, then $\psi_{n}^d(r)>0$. Since $s\rightarrow 1+2s r+r^2$ is increasing, this is the case if and only if $1-2 r+r^2=(r-1)^2<(n-1)^2$, which is (only) if $r<n$.
\end{proof}

Moving back to the specific integrals $I^d_q$ we are interested in,
we combine the latter Corollary and the following approximation result on Bessel functions to show that they are both positive and finite.

\begin{lemma}\cite[Theorem 4, Equation 6]{Krasikov2014}\label{lem:krasikov}
    Let $d\geq 2$. There exist constants $\phi_d>0$ such that
    \begin{equation}\label{eq:krasikov}
        0<
        \sup_{r\geq 0} r^{3/2}\abs{J_\nu(r)-\sqrt{\frac2{\pi r}}\cos\pa{r-\phi_d}}
        <\infty.
    \end{equation}
    In particular, by \eqref{eq:jd} and since $j_d(0)=1$,
    \begin{equation}\label{eq:krasikov2}
        j_d(r)=C_d (r\vee 1)^{-(d-1)/2} \cos(r-\phi_d)+O_d(r^{1-d/2}),
        \quad r\geq 0.
    \end{equation}
\end{lemma}

\begin{lemma}\label{lem:cdq}
    For any $d\geq 2,q\geq 3$ except the single case $d=2,q=4$,
    $I^d_q\in (0,\infty)$ is positive and finite.
\end{lemma}

\begin{proof}
    By Corollary \ref{cor:positivi}, for all $d\geq 2,q\geq 3$ it holds
    \begin{equation*}
        I^d_n=\frac{\rho^d_{n-1}(1)}{(\nu!)^2 4^\nu}=\frac{\psi^d_{n-1}(1)}{(\nu!)^2 4^\nu}>0.
    \end{equation*}
    Concerning finiteness, by \eqref{eq:krasikov2} and \eqref{eq:cdq} we have
    \begin{equation*}
        I^d_q\leq C_{d,q} \int_1^\infty r^{(d-1)(1-q/2)}dr;
    \end{equation*}
    in particular, the integral in \eqref{eq:cdq} is absolutely convergent if $q>\frac{d}{d-1}$. We are left to discuss the cases $d=2,3,q=3$, since we are excluding $d=2,q=4$ (see Remark \ref{rmk:singolarita} below). 
    For $d=2$, by \eqref{eq:krasikov},
    \begin{equation*}
        I^3_2=\int_0^\infty J_0(t)^3tdt=\sqrt{\frac{2}{\pi}}\int_0^\infty \frac{\cos(t)^3}{\sqrt{t}}dt+O(1),
    \end{equation*}
    where, on the right-hand side, the integral converges to a positive number (as it can be shown with elementary passages) and the $O(1)$ accounts for remainder terms involving convergent integrals.
    Finally, $I^3_3=\int_0^\infty t^{-1}\sin(t)^3 dt =\pi/4$
    is easily computed reducing it to Dirichlet's integral.
\end{proof}

\begin{remark}
    The following exact expressions
    \begin{equation*}
        I^d_3=\frac{2}{\pi\sqrt 3}\cdot \frac{12^\nu (\nu!)^4}{(2\nu)!},\,\,d\geq 2, 
        \qquad I^2_5=\frac{\sqrt 5 \Gamma\pa{\frac{1}{15}}\Gamma\pa{\frac{2}{15}}\Gamma\pa{\frac{4}{15}}\Gamma\pa{\frac{8}{15}}}{40\pi^4},
    \end{equation*}
    can be deduced respectively from \cite[Proposition 4.3]{Borwein2016} and \cite[Theorem 5.1]{Borwein2012} (concerning the asymptotics of $\rho^d_n(r)$ as $r\to 0$).
    To the best of our knowledge, no representation is available in the general case.
\end{remark}

\begin{remark}\label{rmk:singolarita}
    In the case $d=2,q=4$ excluded from Lemma \ref{lem:cdq}, one can prove by means of Lemma \ref{lem:krasikov} that $I^2_4=\int_0^\infty J_0(t)^4 tdt=+\infty$,
    coherently with the fact that the density $\rho^3_2(1)=I^4_2=+\infty$
    is not finite.
    This is actually the only case of a density $\rho^d_n$, $d,n\geq 2$, taking value $+\infty$ at some point, as for larger $n$ the regularizing effect of the multiple convolution defining the law of Pearson's random walk prevails (\emph{cf.} again \cite{Borwein2016}).
\end{remark}

%%%%%%%%%%%%%%%%%%%%%%%%%%%%%%%%%%%%%%%%%%%%%
\section{Fluctuation Asymptotics and Integration Domains}\label{sec:polyspectra}

%------------------------------------
\subsection{Variance of Euclidean Polyspectra} \label{ssec:varianceeuclidean}
We first rewrite the variances we are interested in as integrals of a single variable.

\begin{lemma}\label{lem:changeofvareucl}
    Let $R>0$, $B_R=B^E(x_0,R)\subseteq \R^d$ a ball of radius $R$ centered at a (fixed, arbitrary) point $x_0\in \R^d$, and $f\in L^1_{loc}(\R)$; it holds
    \begin{equation}\label{eq:pesoeucl}
        \int_{B_R}\int_{B_R}f(|x-y|)dxdy
        =\int_0^{2R} f(r) W_{d,R}(r) r^{d-1}dr,
    \end{equation}
    where
    \begin{equation}
        W_{d,R}(r)=\omega_{d-1} |B(x,R)\cap B(y,R)|, \quad |x-y|=r,
    \end{equation}
    whose value does not depend on the choice of points $x,y\in \R^d$.
\end{lemma}

The proof is a consequence of the fact that translations on $\R^d$ are isometries of the whole space: it holds
\begin{equation*}
    \int_{\R^d}\int_{\R^d} \chi_{B_R}(x)\chi_{B_R}(y)f(|x-y|)dxdy
    =\int_{\R^d}\int_{\R^d} \chi_{B_R}(x)\chi_{B_R+z}(x)f(|z|)dxdz,
\end{equation*}
from which \eqref{eq:pesoeucl} follows moving to polar coordinates for $z$.
It is easy to observe that $W_{d,R}:[0,\infty)\to[0,\infty)$ is a differentiable function supported by $[0,2R]$, and $W_{d,R}=\omega_{d-1}|B_R|$, $W_{d,R}(2R)=0$.

%The Euclidean part of \cref{thm:main} can now be established by combining the latter observations, the content of the previous section and the approximation result for Bessel functions in \cref{lem:krasikov}.

\begin{proof}[Proof of Theorem \ref{thm:main}, Euclidean part]
    In sight of Lemma \ref{lem:changeofvareucl} it holds,
    \begin{align}\label{eq:vareucl}
        V^{E}_{d,R}(q,\lambda)
        =q! \int_{B_R}\int_{B_R}j_d( \lambda |x-y|)^qdxdy
        =q! \int_0^{2R} j_d(\lambda r)^q W_{d,R}(r) r^{d-1}dr.
    \end{align}
    For $q=2$, \eqref{eq:krasikov2} implies that, as $\lambda\to\infty$,
    \begin{align*}
        V^{E}_{d,R}(2,\lambda)
        =C_d \int_0^{2R} \frac{\cos(\lambda r-\phi_d)^2}{\lambda^{d-1}}W_{d,R}(r)dr+ O_{d,R}\pa{\frac1{\lambda^d}}
        =C_{d,R} \frac{1+o_{d,R}(1)}{\lambda^{d-1}},
    \end{align*}
    the second step following by expanding $2\cos(\theta)^2=1+\cos(2\theta)$
    and applying Riemann-Lebesgue lemma to oscillatory terms.

    We also treat separately the case $d=2,q=4$, since it involves a logarithmic discrepancy with the general case $q\geq 3$. \eqref{eq:krasikov2} implies that, as $\lambda\to\infty$,
    \begin{align*}
        V^{E}_{2,R}(4,\lambda)
        =C_d \int_0^{2R} \frac{\cos(\lambda r-\phi_d)^4}{\lambda^2 r}W_{d,R}(r)dr+ O_{R}\pa{\frac1{\lambda^3}}
        =C_{R} \frac{\log(\lambda)+o_{R}(1)}{\lambda^2},
    \end{align*}
    where, as in the case $q=2$, the last equality follows by $8\cos(\theta)^4=3+4\cos(2\theta)+\cos(4\theta)$ and Riemann-Lebesgue lemma.

    Finally, we consider the general case $q\geq 3$, excluding $d=2,q=4$.
    By \eqref{eq:vareucl},
    \begin{multline*}
        V^{E}_{d,R}(q,\lambda)
        =q! \int_0^{2 R} j_d(\lambda r)^q r^{d-1}
        \pa{\int_r^{2R} -W'_{d,R}(s)ds}dr\\
        =-q! \int_0^{2 R} W'_{d,R}(s)\pa{\int_0^{s}j_d(\lambda r)^q r^{d-1}dr}ds\\
        =\frac{q!}{\lambda^d} \int_0^{2 R} -W'_{d,R}(s)
        \pa{\int_0^{s\lambda}j_d( r)^q r^{d-1}dr}ds,
    \end{multline*}
    integrating by parts in the second step. Since $|W'_{d,R}(r)|$ is bounded
    and supported by $[0,2R]$, the thesis follows by dominated convergence and Lemma \ref{lem:cdq} (implying convergence of the inner integral on the right-hand side as $\lambda\to\infty$).
\end{proof}
%------------------------------------
\subsection{Variance of Hyperspherical Polyspectra}
The forthcoming computations are in fact close analogs of the ones in the previous paragraph, only in a perhaps less familiar geometrical setting. The change of variables of Lemma \ref{lem:changeofvareucl} takes on $S^d$ the following form, with flat translations of $\R^d$ being replaced by isometries of $S^d$.

\begin{lemma}\label{lem:changeofvarsphere}
    Let $R\in [0,\pi]$, $B_R=B^S(x_0,R)\subseteq S^d$ a ball of radius $R$ centered at a (fixed, arbitrary) point $x_0\in S^d$, and $f\in L^1(\R)$; it holds
    \begin{equation}
        \int_{B_R}\int_{B_R}f(d(x,y))d\sigma_d(x)d\sigma_d(y)
        =\int_0^{\pi} f(r) \tilde W_{d,R}(r) \sin(r)^{d-1}dr,
    \end{equation}
    where
    \begin{equation}\label{eq:defWS}
        \tilde W_{d,R}(r)=\omega_{d-1} \sigma_d(B(x,R)\cap B(y,R)), \quad d(x,y)=r,
    \end{equation}
    whose value does not depend on the choice of points $x,y\in S^d$.
\end{lemma}

The proof is conceptually analogous to the one of Lemma \ref{lem:changeofvareucl}. First, let us recall that the isometry group of $S^d$ can be identified with that of linear orthogonal transformations of the ambient $R^{d+1}$, and it acts transitively on $S^d$. The volume $\sigma_d$ and the spherical geodesic distance are invariant under isometries, that is $I_\#\sigma_d=\sigma_d$ and $d(a,b)=d(I(a),I(b))$ for any isometry $I:S^d\to S^d$ and points $a,b\in S^d$.

\begin{proof}
We first observe that for any isometry $I:S^d\to S^d$ it holds
\begin{multline*}
    \sigma_d(B(Ix,R)\cap B(Iy,R))
    =\sigma_d \set{z: d(Ix,z),d(Iy,z)\leq R}\\
    =\sigma_d \set{z: d(x,I^{-1}z),d(y,I^{-1}z)\leq R}
    =\sigma_d \set{Iw: d(x,w),d(y,w)\leq R}\\
    =I^{-1}_\#\sigma_d \set{w: d(x,w),d(y,w)\leq R}=\sigma_d(B(x,R)\cap B(y,R)),
\end{multline*}
therefore the right-hand side only depends on $d(x,y)$, it being a function of $x,y$ invariant under isometries. In particular, $W_{d,R}(r)$ in \eqref{eq:defWS} is well-defined.

For any two points $x,y\in S^d$ we can choose isometries $I_x,I_y:S^d\to S^d$ such that $I_xx_0=x$ and $I_yx_0=y$. Letting $z=I_y I_x^{-1}x_0$, by Fubini-Tonelli theorem it holds
\begin{multline*}
    \int_{S^d}\int_{S^d} \chi\set{x\in B_R}\chi\set{y\in B_R}
    f(d(x,y))d\sigma_d(x)d\sigma_d(y)\\
    =\int_{S^d} d\sigma_d(z)  f(d(x_0,z)) \int_{S^d}  d\sigma_d(x) \chi\set{x\in B_R}\chi\set{I_xz\in B_R},
\end{multline*}
where the inner integral can be rewritten as follows: since $I_x^{-1}x_0=I_x^{-2}x$,
\begin{multline*}
    \int_{S^d} \chi\set{x\in B_R}\chi\set{I_xz\in B_R}  d\sigma_d(x) 
    =\sigma_d \set{x: d(x,x_0),d(I_xz,x_0)\leq R}\\
    =I^2_\#\sigma_d \set{x: d(x_0,I_x^{-1}x_0),d(z,I_x^{-1}x_0)\leq R}\\
    =\sigma_d \set{I_x^{-2}x: d(x_0,I_x^{-1}x_0),d(z,I_x^{-1}x_0)\leq R}\\
    =\sigma_d \set{w: d(x_0,w),d(z,w)\leq R},
\end{multline*}
thus in particular it only depends on $d(x_0,z)$, thanks to the argument at the beginning of the proof. Since the volume element of $S^d$ reduces to
\begin{equation*}
    \int_{S^d} h(d(x_0,x))d\sigma_d(x)= \omega_{d-1}\int_0^\pi h(r) \sin(r)^{d-1}dr, \quad h\in L^1([0,\pi]),
\end{equation*}
for radial functions, combining the above equations concludes the proof.  
\end{proof}

The function $\tilde W_{d,R}:[0,\pi]\to [0,\infty)$ is differentiable, and attains its maximum value in $\tilde W_{d,R}(0)= \omega_{d-1} \sigma_d(B(x_0,R))$ and minimum in $\tilde W_{d,R}(\pi)$ (which, unlike the Euclidean $W_{d,R}(2R)$, does not vanish in general).
We can now proceed to complete the proof of Theorem \ref{thm:main} by reducing ourselves to the computations performed in the Euclidean case thanks to the following:

\begin{lemma}[Hilb's Asymptotic Formula] \label{lem:Hilbs}\cite[Theorem 8.21.12]{Szego1939}
    Let $d\geq 2$, $\ell\in\N$.
    Given $a\in (0,\pi)$, for all $\theta\in (0,a)$
    it holds:
    \begin{equation}
        (\sin\theta )^{\nu} \G_{d,\ell}(\cos \theta)
        = \frac{2^\nu}{\binom{\ell+\nu}{\ell}}
        \pa{\frac{\Gamma(\ell+d/2)}{L^\nu\ell!}
        \pa{\frac{\theta}{\sin \theta}}^{1/2} J_\nu(L\theta)+\delta(\theta)},
    \end{equation}
    where $L=\ell+\frac{d-1}{2}$ and $\delta(\theta) \leq C_{d,a} (\chi_{(0,1/\ell)}(\theta)\theta^{\nu+2}\ell^{\nu}+\chi_{(1/\ell,a)}(\theta)\sqrt{\theta}\ell^{-3/2})$.
\end{lemma}

\begin{corollary}\label{cor:daGaJ}
    Let $d\geq 2$, $\ell\in\N$. For all $a\in (0,\pi)$ it holds, as $\ell\to \infty$,
    \begin{equation}
        \int_{0}^{a} \G_{d,\ell}(\cos r)^q (\sin r)^{d-1} \,dr= (1+o(1))\begin{cases}
            C_d \ell^{1-d} & q=2\\
            C \log(\ell)/\ell^2 & q=4, d=2\\
            I^d_q/ \ell^d & \text{all other } d\geq 2, q \geq 3
        \end{cases}\end{equation}
    with $I^d_q$ being as in Lemma \ref{lem:cdq} and $C,C_d>0$ positive constants.
\end{corollary}

\begin{remark}\label{rmk:pi}
    The limit case $a=\pi$ can be reduced to the one with $a=\pi/2$ splitting the integration domain in half and exploiting the following symmetry property: for all $d\geq 2$,
    \begin{equation*}
        \G_{d,\ell}(t)=(-1)^\ell \G_{d,\ell}(-t), \quad t\in\R,\ell\in\N.
    \end{equation*}
    In particular, the integral in Corollary \ref{cor:daGaJ} identically vanishes if $a=\pi$ and both $q,\ell$ are odd. If this is not the case, then the thesis of Corollary \ref{cor:daGaJ} holds with an additional factor 2 on the right-hand side.
\end{remark}

Corollary \ref{cor:daGaJ} actually collects computations already appeared in \cite{Wigman2010,Marinucci_2011, Marinucci2015}. Because of this and the similarity with computations in the previous paragraph, we only present a sketch of the proof.
\begin{proof}
We first observe that as $\ell\to\infty$,
\begin{equation}\label{eq:gamma}
    \frac{\Gamma(\ell+d/2)}{L^\nu\ell!} =1+o(1),\quad 
    \binom{\ell+\nu}{\ell} =\frac{\ell^\nu}{\nu!}(1+o(1)).
\end{equation}
When $q=2$, Lemma \ref{lem:Hilbs} and \eqref{eq:gamma} imply
\begin{multline*}
    \int_{0}^{a}  \G_{d,\ell}(\cos r)^2 (\sin r)^{d-1} dr=
    \frac{4^\nu (\nu!)^2}{\ell^{2\nu}}(1+o(1))\int_{0}^a J_\nu(Lr) ^2  
    r dr\\
    + O\pa{\frac{1}{\ell^{d-2} }\int_{0}^{a} |J_\nu(Lr)| \delta(r) \sin(r)dr}.
\end{multline*}
The thesis now follows by applying Lemma \ref{lem:krasikov} in the same fashion of the Euclidean case (the remainder term on the right-hand side is proved to be $O(\ell^{-d})$).
The analysis of the case $d=2,q=4$ is similar, we refer again to \cite{Wigman2010} for full details.

As for the remaining cases, by Lemma \ref{lem:Hilbs}, \eqref{eq:gamma} and the definition of $j_d$ in \eqref{eq:jd},
\begin{multline*}
    \int_{0}^{a} \G_{d,\ell}(\cos r)^q (\sin r)^{d-1} dr
   =(1+o(1)) 
   \int_{0}^a \pa{\frac{r}{\sin r}}^{q\nu+q/2-(d-1)} j_d(Lr)^q  
   r^{d-1} dr\\
   + O\left( \frac{1}{\ell^{q\nu}} \int_{0}^{a} (\sin r)^{-q\nu} |J_\nu(Lr)|^{q-1} \delta(r) (\sin r)^{d-1}dr \right) .
\end{multline*}
The remainder term can be shown to be $O(\ell^{-(d+2)\wedge (q(d/2-1/2)-1)})=o(\ell^{-d})$ by means of Lemma \ref{lem:krasikov} (changing variables and splitting the integration domain similarly to the previous case).
The thesis now follows if
\begin{equation*}
    \int_{0}^a \pa{\frac{r}{\sin r}}^{q\nu+q/2-(d-1)} j_d(Lr)^q  
   r^{d-1} dr=(1+o(1))\int_{0}^\infty j_d(r)^q r^{d-1} dr, \quad \ell\to\infty,
\end{equation*}
which can be proved with the same computation exploiting integration by parts we used at the end of the proof of Theorem \ref{thm:main}, Euclidean case,
that is replacing $W_{d,R}$ of that computation with $\chi_{[0,a]}(r)(r/\sin r)^{q\nu+q/2-(d-1)}$.
\end{proof}

\begin{proof}[Proof of Theorem \ref{thm:main}, Hyperspherical part] Let $R\in (0,\pi)$ (the case $R=\pi$ is dealt with similarly, by Remark \ref{rmk:pi}).
    By Lemma \ref{lem:changeofvarsphere},
    \begin{multline*}
        V^{S}_{d,R}(q,\ell)
			=\var\pa{\int_{B_R}H_q(T_{\ell}(x))d\sigma_d(x)}\\
            = q! \int_{B_R}\int_{B_R}\G_{d,\ell}(x,y)^q\,d\sigma_d(x)\,d\sigma_d(y)
            = q! \int_{0}^\pi  \G_{d,\ell}(\cos r)^q \sin(r)^{d-1} \tilde W_{d,R}(r)dr.
    \end{multline*}
    Integrating by parts as in Section \ref{ssec:varianceeuclidean} we rewrite:
    \begin{multline*}
        \int_{0}^\pi \G_{d,\ell}(\cos r)^q (\sin r)^{d-1} \tilde W_{d,R}(r)dr
        =\tilde W_{d,R}(\pi) \int_0^\pi \G_{d,\ell}(\cos r)^q (\sin r)^{d-1} dr\\
        -\int_0^\pi \pa{\int_0^r \G_{d,\ell}(\cos s)^q (\sin s)^{d-1} ds} \tilde W_{d,R}'(r)dr.
    \end{multline*}
    The thesis now follows from Corollary \ref{cor:daGaJ}. 
    We observe in particular that when we are not in the cases $q=2$ nor in the one $d=2,q=4$, by Corollary \ref{cor:daGaJ} and dominated convergence we have
    \begin{multline*}
        \ell^d \int_{0}^\pi \G_{d,\ell}(\cos r)^q (\sin r)^{d-1} \tilde W_{d,R}(r)dr\\
        \xrightarrow{\ell\to\infty}
        I^d_q \tilde W_{d,R}(\pi) -I^d_q \int_0^\pi \tilde W_{d,R}(r)'dr
        =I^d_q \tilde W_{d,R}(0).\qedhere
    \end{multline*}
\end{proof}

\begin{acknowledgements}
F. G. and A. P. T. acknowledge support of INdAM through the INdAM-GNAMPA Project CUP\_E55F22000270001. L.M. acknowledges support of the Luxembourg National Research Fund PRIDE17/1224660/GPS. F. G. wishes to thank the University of Luxembourg for the warm hospitality during part of the preparation of the present work.
\end{acknowledgements}

\bibliographystyle{plain}

\end{document}